\documentclass[leqno]{amsart}
\usepackage{amsfonts}
\usepackage{amsmath}
\usepackage{amsthm}
\usepackage{marginnote}
\usepackage{multirow}
\usepackage{bbm}
\usepackage{color,url}
\usepackage{mathrsfs}
\usepackage{mathtools,pdfsync}
\mathtoolsset{showonlyrefs}

\allowdisplaybreaks

\numberwithin{equation}{section}

\theoremstyle{plain}
\newtheorem{theorem}{Theorem}[section]
\newtheorem{lemma}[theorem]{Lemma}

\newtheorem{definition}[theorem]{Definition}

\theoremstyle{remark}
\newtheorem*{remark}{Remark}

\def\nn{\nonumber}
\def\RZ{\mathbb{R}^{\infty}_{+,\mathbb{Z}}}
\def\dZ{d_{\infty,\mathbb{Z}}}

\newcommand{\bbZ}{\mathbb{Z}}
\newcommand{\R}{\mathbb{R}}
\newcommand{\bfS}{\mathbb{S}}
\newcommand{\bfM}{\mathbb{M}}
\newcommand{\bfC}{\mathbb{C}}
\newcommand{\bfO}{\mathbb{O}}
\newcommand{\bfX}{\mathbf{X}}
\newcommand{\bfZ}{\mathbf{Z}}
\newcommand{\bfT}{\mathbf{T}}
\newcommand{\bfTm}{\mathbf{T}^{m}}
\newcommand{\bfTinf}{\mathbf{T}^{\infty}}

\newcommand{\bfF}{\mathbb{F}}
\newcommand{\bfG}{\mathbb{G}}
\newcommand{\bfz}{\mathbf{z}}
\newcommand{\zero}{\mathbf{0}_{\infty}}

\newcommand{\C}{\mathscr{C}}
\newcommand{\scrC}{\mathscr{C}}
\newcommand{\scrS}{\mathscr{S}}
\newcommand{\M}{\mathbb{M}}
\newcommand{\MO}{\mathbb{M}_{\bfO}}
\newcommand{\ud}{\mathrm{d}}
\newcommand{\vep}{\varepsilon}
\newcommand{\bep}{\boldsymbol{\epsilon}}

\newcommand{\mutj}{\mu_t^{(j)}}
\newcommand{\muj}{\mu^{(j)}}
\newcommand{\mutz}{\mu_t^{(0)}}
\newcommand{\muz}{\mu^{(0)}}

\title[Hidden regular variation of moving average processes]{Hidden regular variation of moving average processes with heavy-tailed innovations}



\author[Resnick ]{Sidney I. Resnick}
\address{Sidney I. Resnick\\School of ORIE, Cornell University,
Ithaca, NY 14853} \email{sir1@cornell.edu}

\author[Roy ]{Joyjit Roy}
\address{Joyjit Roy\\School of ORIE, Cornell University,
Ithaca, NY 14853} \email{jr653@cornell.edu}

\thanks{S. Resnick and J. Roy were supported by Army MURI grant W911NF-12-1-0385 to Cornell University.} 

\begin{document}
\begin{abstract}
We look at joint regular variation properties of MA($\infty$)
processes of the form $\mathbf{X} = (X_k, k \in \bbZ)$ where $X_k =
\sum_{j=0}^{\infty} \psi_j Z_{k-j}$ and the sequence of random
variables $(Z_i, i \in \bbZ)$ are i.i.d.~ with regularly varying
tails. We use the setup of $\mathbb{M}_{\mathbb{O}}$-convergence and
obtain hidden regular variation properties for $\mathbf{X}$ under
suitable summabality conditions on the constant coefficents $( \psi_j
: j \geq 0 )$. Our approach emphasizes continuity properties of
mappings and produces regular variation in sequence space. 
\end{abstract}

\maketitle

\section{Introduction}

The purpose of this paper is to obtain joint regular variation properties of moving average processes of the form
\begin{align*}
 X_k =  \sum_{j=0}^{\infty} \psi_j Z_{k-j}, \;k \in \bbZ,
\end{align*}
where $Z_i$ are i.i.d.~ nonnegative heavy-tailed random variables and
$\psi_j$ are constant nonnegative coefficients. The study of tail
behavior of such processes has a long history. Early studies of the one-dimensional
case with constant coefficients are
\cite{rootzen:1978, cline:1983a, davis:resnick:1985, rootzen:1986}; see
also accounts in \cite{resnickbook:2008, brockwell:davis:1991}.
The $d$-dimensional results as well as results for moving
average processes with random coefficients can be found in
\cite{resnick:willekens:1990, hult:samorodnitsky:2008,
  mikosch:samorodnitsky:2000}. Joint regular variation properties of
the MA($\infty$) process were obtained in
\cite{davis:resnick:1985}. Many of these studies emphasized
 finding proper summability conditions for the coefficient sequence
which forces the extremal properties of the process to be determined
by the tail behavior of the innovation sequence. In this paper we use
a fairly strong summabilty assumption on the coefficent sequence and
concentrate on using continuity arguments to obtain joint regular variation properties of
the entire sequence as a random element of the space of double-sided
sequences. 

Traditionally multivariate regular variation properties of $d$-dimensional random vectors have been expressed using the theory of vague convergence where limit measures are finite on compact sets. To make extremal sets, such as sets which are unbounded above, compact, the approach is to compactify a locally compact space such as $[0,\infty)^2$ by adding lines through $\infty$ to obtain $[0,\infty]^2$ and then restrict the class of sets on which the limit measure has to be finite by removing a point such as $(0,0)$ to obtain $[0,\infty]^2 \setminus \{(0,0)\}$. For more details on vague convergence see \cite{kallenberg:1983} while look at \cite{resnick:1987, resnick:2007} for applications of vague convergence to extreme value theory. There are a few systemic problems inherent to using vague convergence theory such as, having to worry about lines through $\infty$ as well as points of uncompactification when dealing with continuous mapping arguments. Further the theory is limited to locally compact spaces.

An alternate framework for dealing with tail behavior of general random
elements, the theory of $\bfM_{\bfO}$-convergence was developed in
\cite{hult:lindskog:2006} which applies the theory of
$w_{\#}$-convergence (\cite{daley:vere-jones:vol1:2007}) to obtain
a framework which lends itself nicely to dealing with regular
variation on any complete, separable metric space with a point
removed. The theory was further extended to allow consideration of
spaces with a general closed cone removed in
\cite{lindskog:resnick:2013}. The main attraction of any such theory
lies in powerful mapping theorems and their use to obtain results
about transformations and functionals (see \cite{hult:lindskog:2005,
  hult:lindskog:2007}). In this paper we prove that $\mathbf{X} =
(X_k, k \in \bbZ)$ is regularly varying as an element of $\RZ
\setminus \{ \zero \}$ using such a mapping argument. 

Another aspect of multivariate regular variation that is relevant to our paper is the concept of hidden regular variation which was first developed in \cite{resnick:2002, maulik:resnick:2004}. As a simple example, we can look at 2 concurrent regular variation properties of an i.i.d.~ Pareto(1) pair of random variables $(X_1,X_2)$. Observe that for $x_1,x_2 \geq 0$,
\begin{align*}
tP[(X_1 > tx_1,X_2 > tx_2] \to \left\{ \begin{array}{clc} (x_1 \vee x_2)^{-1} &\text{ if } &x_1 \wedge x_2 = 0 \\ 0 &\text{ o.w. } &  \end{array} \right.
\end{align*}
while 
\begin{align*}
tP[(X_1 > t^{1/2}x_1,X_2 > t^{1/2}x_2] \to  (x_1x_2) ^{-1} \text{ if } &x_1 \wedge x_2 > 0 .
\end{align*}
Here the second regular variation property of the pair $(X_1,X_2)$ is
only applicable on a part of the state space obtained by removing the
support of the limit measure in the first regular variation property
and then using a scaling function that goes to $\infty$ slower than
$t$. The second property was hidden by the coarse
scaling used to obtain convergence to a non-zero measure in the first
case. The theory of $\bfM_{\bfO}$-convergence has already been
fruitfully applied to prove the existence of hidden regular variation
(\cite{das:mitra:2013}). In this paper we obtain an infinite
sequence of hidden regular variation properties for the finite moving
average process as an element of $\RZ$.  

In Section 2, we define $\bfM_{\bfO}$-convergence and collect relevant
results about the theory as well as the definition of regular
variation of a random variable in this framework. In Section 3.1, we
restate results about regular variation of i.i.d.~ heavy-tailed
sequences obtained in \cite{lindskog:resnick:2013} which will form the
basis for proving our results. In Section 3.3 we prove the existence
of hidden regular variation for the MA($m$) process before proving our
main theorem in Section 3.4. Due to technical considerations, proving
a hidden regular variation property for the MA($\infty$) sequence has
not yet been possible but the authors are working towards achieving
that end; we instead prove hidden regular variation for finite order
moving averages. 
Still our main result about the joint regular variation of
the entire sequence not only serves as a nice demonstration of the
power of continuous mapping theorems in the $\bfM_{\bfO}$ framework
but will serve as a building block for obtaining further results
through the use of other mappings and functionals to the sequence
space. 

\section{Basics of $\mathbb{M}_{\mathbb{O}}$-convergence and regular variation of measures}

In this section we define the framework for $\mathbb{M}_{\mathbb{O}}$-convergence and collect basic results that will be useful later. For more details and proofs see Sections 2 and 3 in \cite{lindskog:resnick:2013}.

\subsection{$\MO$-convergence}

Let $(\mathbb{S}, d)$ be a complete separable metric space. Fix a closed set $\bfC \subset \bfS$ and set $\bfO=\bfS\setminus
\bfC$,  The subspace $\bfO$ is a metric subspace of $\bfS$ in the relative topology with $\sigma$-algebra {$\scrS_{\bfO}=\scrS(\bfO)=\{A : A \subset \bfO, A \in \scrS\}$}. 

Let $\C_b$ denote the class of real-valued, non-negative, bounded and continuous functions on $\bfS$, and let $\M_b$ denote the class of finite Borel
measures on $\scrS$. A basic neighborhood of $\mu \in \M_b$ is a set of the form $\{\nu \in \M_b :|\int f_i \; \ud\nu -\int f_i \; \ud\mu| < \vep, i = 1,\dots,k\}$, where $\vep > 0$ and $f_i \in \C_b$ for $i=1,\dots,k$. This equips $\M_b$ with the weak topology and convergence $\mu_n \to \mu$ in $\M_b$ means  $\int f \ud \mu_n \to \int f \ud \mu$ for all $f \in \C_b$. See, for example, Sections 2 and 6 in \cite{billingsley:1999} for details. 

Let $\C(\bfO)$ denote the real-valued, non-negative, bounded and continuous functions $f$ on $\bfO$ such that for each $f$ there exists $r>0$ such that $f$ vanishes on $\bfC^r$; we use the notation $\bfC^r=\{x\in \bfS:d(x,\bfC)<r\}$, where $d(x,\bfC)=\inf_{y\in \bfC}d(x,y)$. Similarly, we will write $d(A,\bfC)=\inf_{x \in A,\,y\in \bfC}d(x,y)$ for $A\subset\mathbb{S}$. We say that a set $A\in\scrS_{\bfO}$ is bounded away from $\bfC$ if $A \subset \bfS \setminus\bfC^r$ for some $r > 0$ or equivalently $d(A,\bfC)>0$. So $\C(\bfO)$ consists of non-negative continuous functions whose supports are bounded away from $\bfC$. Let $\MO$ be the class of Borel measures on $\bfO$ whose restriction
to $\bfS \setminus\bfC^r$ is finite for each $r > 0$. When convenient, we also write $\mathbb{M}(\mathbb{O})$ or $\mathbb{M}(\mathbb{S}\setminus \bfC)$.  A basic neighborhood of $\mu \in \MO$ is a set of the form $\{\nu \in \MO :|\int f_i \ud \nu - \int f_i \ud \mu| < \vep, \,i = 1,\dots,k\}$, where
$\vep > 0$ and $f_i \in \C_{\bfO}$ for $i=1,\dots,k$. Convergence $\mu_n\to\mu$ in $\MO$ is convergence in the topology defined by this base. As the next theorem shows, it actually suffices to consider the class of uniformly continuous functions in $\C(\bfO)$.

\begin{theorem}\label{portthm}
Let $\mu,\mu_n \in \MO$.Then the following statements are equivalent.
\\
(i) $\mu_n \to \mu$ in $\MO$ as $n\to\infty$.
\\
(ii)  $\int f \ud\mu_n \to \int f \ud \mu$ for each $f \in \C(\bfO)$ which is also uniformly continuous on $\bfS$. 
\\
{(iii)} $\mu_n^{(r)}\to\mu^{(r)}$ in $\M_b(\bfS \setminus \bfC^r)$ for all $r>0$ such that $\mu(\partial \, \bfS \setminus \bfC^r) = 0$  where $\mu^{(r)}$ denote the restriction of $\mu$ to $\bfS \setminus \bfC^r$.
\\
\end{theorem}

Continuous mapping theorems will play an important role in extending the regular variation property of the innovation sequence to that of the actual moving average sequence. Here we state one version that will prove useful to us. Consider another separable and complete metric space $\bfS'$ and let $\bfO',\scrS_{\bfO'},\bfC',\M_{\bfO'}$ have the same meaning relative to the space $\bfS'$ as do $\bfO,\scrS_{\bfO},\bfC,\M_{\bfO}$ relative to $\bfS$.
\begin{theorem}\label{mappingtheorem}
Suppose $h:\bfS \mapsto \bfS'$ is uniformly continuous and $\bfC':=h(\bfC)$ is closed in $\bfS'$. Then $\hat h: \bfM_{\bfO} \mapsto \bfM_{\bfO'} $ defined by $\hat h(\mu)=\mu {\circ } h^{-1}$ is continuous. 
\end{theorem}

\subsection{Regular variation of measures}
The usual notion of regular variation involves comparisons along a ray and requires a concept of scaling or multiplication. Given any real number $\lambda>0$  and any $x\in \bfS$, we assume there exists a mapping $(\lambda,x)\mapsto \lambda x$ from $(0,\infty) \times \bfS$ into $\bfS$ 
satisfying:
\begin{itemize}
\item the mapping $(\lambda,x)\mapsto \lambda x$ is continuous,
\item $1x=x$ and $\lambda_1(\lambda_2 x)=(\lambda_1\lambda_2)x$.
\end{itemize}
These two assumptions allow definition of a cone $\bfC \subset \mathbb{S}$ as a set satisfying $x\in \bfC$ implies $\lambda x \in \bfC$ for any $\lambda>0$. For this section, fix a closed cone $\bfC \subset \bfS$
and then $\bfO:=\bfS\setminus \bfC$ is an open cone. Also assume that 
\begin{itemize}
\item $d(x,\bfC) < d(\lambda x,\bfC)$ if $\lambda>1$ and $x\in\bfO$.
\end{itemize}

Recall from, for example, \cite{bingham:goldie:1987} that a positive measurable function $c$ defined on $(0,\infty)$ is regularly varying
with index $\rho \in \R$ if $\lim_{t\to\infty}c(\lambda t)/c(t)=\lambda^{\rho}$ for all $\lambda > 0$. Similarly, a sequence $\{c_n\}_{n\geq 1}$ of positive numbers is regularly varying with index $\rho \in \R$ if $\lim_{n\to\infty}c_{[\lambda  n]}/c_n=\lambda^{\rho}$ for all $\lambda > 0$.  Here $[\lambda n]$ denotes the integer part of $\lambda n$.  

\begin{definition}\label{rvdefseq}
A sequence $\{\nu_n\}_{n \geq 1}$ in  $\MO$ is regularly varying 
if there exists an increasing sequence $\{c_n\}_{n \geq 1}$ of positive numbers which is regularly varying and a nonzero $\mu \in \MO$ such that $c_n\nu_n \to \mu$ in $\MO$ as $n\to\infty$.
\end{definition}

We now define regular variation for a single measure in $\MO$ as well as an equivalent formulation that is more pleasing to handle algebraically.

\begin{definition}\label{rvdefmeas}
A measure $\nu\in\MO$ is regularly varying if the sequence $\{\nu(n\cdot)\}_{n\geq 1}$ in $\MO$ is regularly varying  or equivalently there exist a nonzero $\mu \in \MO$ and an increasing function $b$ of such that $t\nu(b(t)\cdot)\to\mu(\cdot)$ in $\MO$ as $t\to\infty$. Similarly we say that a random variable $\mathbf{Y}$ taking values in $\bfS$ is regularly varying if the associated probability measure is regularly varying, ie if $P[\mathbf{Y} \in \cdot]$ is regularly varying. 
\end{definition}
We will refer to the function $b$ as the scaling function correspoding to the regularly varying measure $\nu$ on $\bfM_{\bfO}$.
\section{Main results}

\subsection{Hidden regular variation for i.i.d.~ heavy tailed sequences}

From here on we will look at $\mathbb{S} = \RZ$ where $\RZ$ is defined to be the space of all double-sided sequences of non-negative real numbers ie $\RZ = \{\mathbf{x} = (x_i, i \in \mathbb{Z}) : x_i \geq 0 \}$ equipped with the metric $\dZ$ defined as 
\begin{align}\label{dinftydef}
\dZ(\mathbf{x},\mathbf{y}) = \sum\limits_{i=-\infty}^{\infty} \frac{ |x_i - y_i | \wedge 1}{2^{|i| + 1}}.
\end{align}
The concept of multiplication will be given by the standard pointwise multiplication of a sequence by a real number. 

Observe that convergence in this metric is equivalent to convergence of all finite dimensional sequences, i.e.~ $\dZ(\mathbf{x^n},\mathbf{x}) \to 0$, if and only if for any $M \in \mathbb{Z}_+$, the sequences $(x_i^n, |i| \leq M)$ converge pointwise to $(x_i, |i| \leq M)$ in $\mathbb{R}^{2M+1}$. Further observe that $(\RZ,\dZ)$ as a metric space is homeomorphic to $(\mathbb{R}_+^{\infty},d_{\infty})$ where $\mathbb{R}_+^{\infty} = \{\mathbf{x} = (x_i, i \in \mathbb{N}) : x_i \geq 1 \}$ and $d_{\infty}(\mathbf{x},\mathbf{y}) = \sum_{i=1}^{\infty} \frac{ |x_i - y_i | \wedge 1}{2^i}$.

Define $\mathbf{0}_{\infty}$ to be the sequence with all components 0 in $\RZ$ and $\mathbf{e}_i$ to be the sequence in $\RZ$ with the $i$'th component 1 and all other components 0. Further define
\begin{equation}\label{cjdef}
\begin{aligned}
\mathbb{C}_{=j} &= \{ \mathbf{x} \in \RZ : \sum\limits_{i=-\infty}^{\infty} \bep_{x_i} ((0,\infty)) = j \} \text{ for all } j \geq 1, \text { and }\\
\mathbb{C}_{\leq j} &= \{ \mathbf{x} \in \RZ : \sum\limits_{i=-\infty}^{\infty} \bep_{x_i} ((0,\infty)) \leq j \} \text{ for all } j \geq 0.
\end{aligned}
\end{equation}
Define $\mathbb{O}_j = \RZ \setminus \mathbb{C}_{\leq j-1}$ for $j \geq 1$. 

Let $\mathbf{Z} = (Z_i, i\in \mathbb{Z})$ be i.i.d.~ random variables in $\mathbb{R}_+$ with regularly varying tails with index $\alpha > 0$ ie 
\begin{align}\label{iidseqdef1}
\lim_{t \to  \infty}\frac{ P[Z_0 > tz]}{P[Z_0 >t]}= z^{-\alpha} \text{ for all } z>0.
\end{align}
or equivalently for some regularly varying sequence $b(\cdot)$,  
\begin{align}\label{iidseqdef2}
\lim_{t \to  \infty} tP[Z_0 > b(t)z] = z^{-\alpha} \text{ for all } z>0.
\end{align}
With this setup, we can restate Theorem 4.2 in \cite{lindskog:resnick:2013} as a statement about the space $\RZ$ and the sequence of i.i.d.~ random variables $\mathbf{Z} \in \RZ$. Define for each $j \geq 0$,
\begin{equation}\label{mujdef}
\begin{aligned}
&\mutj (\cdot) = tP(\bfZ/b(t^{1/(j+1)}) \in \cdot) \text{ and } \\
&\muj (\cdot) = \sum_{(i_1,\dots,i_{j+1})}\int \mathbbm{1} \left\{\sum_{k=1}^{j+1}z_{k} \mathbf{e}_{i_k}\in \cdot \right\} \nu_\alpha (\ud z_1)\dots \nu_\alpha (\ud z_{j+1}),
\end{aligned}
\end{equation}
where $\nu_\alpha(x,\infty) = x^{-\alpha}$ and the indices $(i_1,\dots,i_{j+1})$ run through the ordered subsets of size $j+1$ of $\bbZ$.
\begin{theorem}\label{rinftyconvergence}
For every $j \geq 0$, $\mutj \to \muj$ in $\mathbb{M}(\mathbb{O}_j)$. The measure $\muj$ concentrates on $\mathbb{C}_{\leq j+1} \setminus \mathbb{C}_{\leq j} = \mathbb{C}_{= j+1}$ and has the alternative form
\begin{align}\label{mualtdef}
& \muj(\ldots, \ud z_1, \ud z_0, \ud z_{-1}, \ldots) \\
&\qquad=  \sum_{(i_1,\dots,i_{j+1})}  \left( \prod_{k \notin \{i_1,\dots,i_{j+1}\} }\bep_0 (\ud z_k) \right)    \left( \prod_{k \in \{i_1,\dots,i_{j+1}\} }\nu_{\alpha} (\ud z_k) \right). \nn
\end{align}
\end{theorem}


\subsection{Definition of the MA($\infty$) process and framework of proof of the main result}
Let $(\psi_j, j \geq 0)$ be a sequence of non-negative constants with 
\begin{itemize}
\item[(A1)]$\psi_0 > 0$ and
\item[(A2)] for some $\delta < \alpha \wedge 1$, $\sum_{j=0}^{\infty} \psi_j^{\delta} < \infty$.
\end{itemize}
Observe that Assumption (A2) implies the following:
\begin{itemize}
\item[(C1)] $\sum_{j=0}^{\infty} \psi_j < \infty$,
\item[(C2)] $\sum_{j=0}^{\infty} \psi_j^{\alpha} < \infty$,
\item[(C3)] for any $k \in \bbZ$, the sequence $\sum_{j=0}^{\infty} \psi_j Z_{k-j}$ converges almost surely, and
\item[(C4)] for any $x>0$ and $k \in \bbZ$,
\begin{align*}
\lim_{N \to \infty} \limsup_{t \to \infty} t P \left[\sum_{j>N} \psi_j Z_{k-j} > b(t)x \right] = 0,
\end{align*}
\end{itemize}
where $\mathbf{Z} = (Z_i, i \in \bbZ)$ is defined as in \eqref{iidseqdef2}.
(C1) and (C2) are easy to see while proofs of (C3) and (C4) can be found in \cite{resnick:1987} (Section 4.5, especially Lemma 4.24)  and \cite{cline:1983a,cline:1983b}.
\begin{remark}
Assumption (A2) is not the weakest condition known in the literature that implies (C1 - C4). See \cite{wang:tang:2006, hult:samorodnitsky:2008, resnick:willekens:1990, mikosch:samorodnitsky:2000} for different summability assumptions on the sequence $(\psi_j, j \geq 0)$ as well as a treatment of moving average processes with random coefficients and heavy tailed innovations.
\end{remark}

For $k \in \bbZ$, define
\begin{align}\label{mainftydef}
X_k = \sum_{j=0}^{\infty} \psi_j Z_{k-j}.
\end{align}
(C3) ensures that $\mathbf{X} = (X_k, k \in \bbZ)$ is a well defined
sequence of random variables in $\RZ$. If we define the map
$\mathbf{T}^{\infty} : \RZ \ni \mathbf{z} = (z_i, i \in \bbZ) \mapsto
\bfT^{\infty}(\mathbf{z}) = (T_k^{\infty}(\mathbf{z}), k \in \bbZ) \in
\RZ$ where  
\begin{align}\label{tinftydef}
T_k^{\infty}(\mathbf{z}) = \sum_{j=0}^{\infty} \psi_j z_{k-j},
\end{align}
then we have that 
\begin{align*}
\bfX = \bfT^{\infty}(\bfZ).
\end{align*}
This leads us to suspect that regular variation properties can be
obtained from Theorem \ref{rinftyconvergence} using a continuous
mapping argument. But unfortunately, the map $\bfT^{\infty}$, even
though well-defined $P$-almost surely,  is nowhere continuous on
$\RZ$. This forces us to use a truncation argument as detailed in the
sequel by using a sequence of maps which map $\bfZ$ to the partial
sums of the infinite sums that make up the elements of $\bfX$ and then
using a Slutsky-style approximation. The details are 
technical as we are dealing with infinite measures. 

\subsection{Hidden regular variation of the MA($m$) process}
For every $m \geq 0$, define the random variable  $\mathbf{X}^m = (X_k^m, k \in \bbZ) \in \RZ$ where
\begin{align}\label{xmdef}
X_k^m = \sum_{j=0}^{m} \psi_j Z_{k-j}.
\end{align}
Similar to \eqref{tinftydef} we can define for every $m \geq 0$ the map $\mathbf{T}^{m} : \RZ \ni \mathbf{z} = (z_i, i \in \bbZ) \mapsto \bfT^{m}(\mathbf{z}) = (T_k^{m}(\mathbf{z}), k \in \bbZ) \in \RZ$ where 
\begin{align}\label{tmdef}
T_k^{m}(\mathbf{z}) = \sum_{j=0}^{m} \psi_j z_{k-j}.
\end{align}
Again we have that $\bfX^m = \bfT^{m}(\bfZ)$. However the map $\bfT^m$ is well-behaved enough for us to use Theorem \ref{mappingtheorem}. We first prove two preliminary lemmas to enable the use of said theorem, which will lead to our main result about the MA($m$) processes.

\begin{lemma}\label{lem:tmcontinuity}
For every $m \geq 0$, the map $\bfT^m$ is uniformly continuous.
\end{lemma}
\begin{proof}
Fix $m \geq 0$ and $\epsilon > 0$. Let $M>0$ be such that $2.2^{-M} < \vep/2$. Take $\mathbf{x} = (x_i , i \in \bbZ), \mathbf{y} = (y_i , i \in \bbZ) \in \RZ$. Then
\begin{align}
\dZ(\bfT^m(\mathbf{x}),\bfT^m(\mathbf{y})) &< \sum_{|i| < M}  \frac{ |\sum_{j=0}^{m} \psi_j x_{i-j}  - \sum_{j=0}^{m} \psi_j y_{i-j} | \wedge 1}{2^{|i|+1}} + \vep/2 \nn \\
\label{lem:tmcontinuity:eq1}
&\leq 2 (\sum_{j=0}^{m} \psi_j) (\bigvee_{|i| < M+m}|x_i-y_i|) + \vep/2 
\end{align}
Let $\delta < (\sum_{j=0}^{m} \psi_j)  \frac{\vep}{4} 2^{-(M+m)}$ and assume that $\dZ(\mathbf{x},\mathbf{y}) < \delta$. Then from \eqref{dinftydef}, we get that 
\begin{align}
 \bigvee_{|i| < M+m}|x_i-y_i| < 2^{(M+m)} \delta <  (\sum_{j=0}^{m} \psi_j)  \frac{\vep}{4}. \nn
\end{align}
Then using \eqref{lem:tmcontinuity:eq1} we get that
\begin{align*}
\dZ(\bfT^m(\mathbf{x}),\bfT^m(\mathbf{y})) < \vep.
\end{align*}
\end{proof}

\begin{lemma}\label{lem:tmcjclosed}
For every $m \geq 0$ and $j \geq 0$,  $\bfT^m(\bfC_{\leq j})$ is closed, where $\bfC_{\leq j}$ is defined as in \eqref{cjdef}.
\end{lemma}
\begin{proof}
Fix $m \geq 0$ and observe that for $j=0$ , $\bfT^m(\bfC_{\leq 0})= \bfT^m ( \{ \mathbf{0}_{\infty} \}) = \{ \mathbf{0}_{\infty} \}$ which is trivially closed. This behooves us to prove the result by induction. So assume that the result holds for $j < J$. Take $\mathbf{z^n} \in \bfC_{=J}$  such that $\mathbf{z}^n \to \mathbf{z} \in \RZ$. It is enough to assume this as $\bfC_{\leq J} = \bfC_{\leq J-1} \cup \bfC_{=J}$.  Further assume that we have sequences $\lambda^n_1, \lambda^n_2, \ldots, \lambda^n_J > 0$ and $ i_1^n< i_2^n< \ldots< i_J^n \in \bbZ$ such that 
\begin{align*}
\mathbf{z}^n &= \sum_{k=1}^{J} \lambda^n_k \mathbf{e}_{ i_k^n} \text{ and } \\
\bfT^m(\mathbf{z}^n) &= \sum_{l=0}^{m} \sum_{k=1}^{J} \psi_l \lambda^n_k \mathbf{e}_{ i_k^n - j}.
\end{align*}
Observe that if $i_J^n \to -\infty$ along some subsequence ${n_q}$, then the limit of any finte dimensional subequence of $\bfT^m(\mathbf{z}^{n_q})$ is the same as the finite dimensional subsequential limit of $ \bfT^m (\sum_{k=1}^{J-1} \lambda^{n_q}_k \mathbf{e}_{ i_k^{n_q}})$. Since limits of a sequence in $\RZ$ is determined by the limits of the finite dimensional subsequences, we have by the induction hypothesis that $\mathbf{z} \in \bfT^m(\bfC_{\leq J -1}) \subset \bfT^m(\bfC_{\leq J})$. A similar argument shows that if $i_1^n \to \infty$ along some subsequence then $\mathbf{z} \in \bfT^m(\bfC_{\leq J})$. So we must have that $( i_1^n, i_2^n, \ldots, i_J^n)$ are contained in some bounded set and so they must equal some $( i_1, i_2, \ldots, i_J)$ infintely often where $i_1 <  i_2 < \ldots <  i_J$. Without loss of generality we may now assume that 
\begin{align}
\mathbf{z}^n &= \sum_{k=1}^{J} \lambda^n_k \mathbf{e}_{ i_k} \text{ and } \nn \\
\label{lem:tmcjclosed:eq1}
\bfT^m(\mathbf{z}^n) &= \sum_{l=0}^{m} \sum_{k=1}^{J} \psi_l \lambda^n_k \mathbf{e}_{ i_k - j}.
\end{align}
Since $(\bfT^m(\mathbf{z}^n))_{i_J - m} = \psi_0 \lambda_J^n$ converges and by assumption (A1), $\psi_0 >0$, we must have that $\lambda_J^n \to \lambda $ for some $\lambda \geq 0$. This  leads to $\bfT^m(\sum_{k=1}^{J-1} \lambda_k^n \mathbf{e}_{ i_k}) \to \mathbf{z} - \bfT^m(\lambda \mathbf{e}_{i_J})$. But the induction hypothesis now gives us that $\mathbf{z} - \bfT^m(\lambda \mathbf{e}_{i_J}) \in \bfT^m(\bfC_{ \leq J-1})$. Thus we get that $\mathbf{z} \in \bfT^m(\bfC_{\leq J})$ proving the induction step.
\end{proof}

A quick application of Theorem \ref{mappingtheorem} now gives us the follwing result.
\begin{theorem}\label{tmconvergence}
For every $m \geq 0$ and $j \geq 0$, $\mutj \circ (\bfT^m )^{-1} \to \muj \circ (\bfT^m )^{-1}$ in $\bfM(\RZ \setminus \bfT^m(\bfC_{\leq j}))$ or equivalently
\begin{align*}
 t P \left[ \bfX^m/b(t^{1/(j+1)}) \in \cdot \right] \to \sum_{(i_1,\dots,i_{j+1})}\int \mathbbm{1} \left\{ \bfT^m(\sum_{k=1}^{j+1}z_{k} \mathbf{e}_{i_k})\in \cdot \right\} \nu_\alpha (\ud z_1)\dots \nu_\alpha (\ud z_{j+1}) .
\end{align*}
\end{theorem}

\begin{remark}
Observe that Theorem \ref{tmconvergence} implies an infinitude of regular variation properties for $\mathbf{X}^m$.

For example,  for $j=0$,
\begin{align*}
 t P \left[ \bfX^m/b(t) \in \cdot \right] \to \nu^{m,(0)} (\cdot) \text{ in } \bfM(\RZ \setminus \bfT^m( \{ \mathbf{0}_{\infty} \} )) = \bfM(\RZ \setminus \{\mathbf{0}_{\infty}\}),
\end{align*}
where
\begin{align*}
\nu^{m,(0)} (\cdot) =  \sum_{i=-\infty}^{\infty} \int \mathbbm{1} \left\{ \bfT^m(z_i \mathbf{e}_i)\in \cdot \right\} \nu_\alpha (\ud z_i).
\end{align*}
From the above it is clear that $\nu^{m,(0)}$ is a non-zero measure, it is finite on subsets of $\RZ$ bounded away from $\mathbf{0}_{\infty}$ and its support is on $\bfT^m(\bfC_{=1})$. Thus $\mathbf{X}^m$ is regularly varying on $\RZ \setminus \{\zero\}$ with scaling function $b(\cdot)$ and limit measure $\nu^{m,(0)}$. Using \eqref{mualtdef} assuming $\psi_j > 0$ for all $j \leq m$, we have the following alternate and slightly more  illuminating formulation for $\nu^{m,(0)}$, namely
\begin{align}
& \nu^{m,(0)}(\ldots, \ud z_1, \ud z_0, \ud z_{-1}, \ldots) \nn\\
&\qquad=  \sum_{i=-\infty}^{\infty}  \left( \prod_{k < i }\epsilon_0 (\ud z_k) \right)    \left( \prod_{i \leq k \leq i +m} \nu_{\alpha}\left (\frac{\ud z_k}{\psi_{k-i}} \right) \right) \left( \prod_{k > i + m }\bep_0 (\ud z_k) \right). \nn
\end{align}
 Further for any $k \in \bbZ$, we have that 
\begin{align} \label{xkmlimit}
t P[ X_k^m > b(t)x] \to  \left( \sum_{l=0}^{m} \psi_l^{\alpha} \right) x^{-\alpha} \text{ for } x>0.
\end{align}

Similarly for $j=1$, 
\begin{align*}
 t P \left[ \bfX^m/b(t^{1/2}) \in \cdot \right] \to \nu^{m,(1)} (\cdot) \text{ in } \bfM(\RZ \setminus \bfT^m( \bfC_{\leq 1})) ,
\end{align*}
where $\nu^{m,(1)}$ is a non-zero measure on $\RZ \setminus \bfT^m( \bfC_{\leq 1})$ with support $\bfT(\bfC_{=2})$. So $\bfX^m$ is also regularly varying on $\RZ \setminus \bfT^m( \bfC_{\leq 1})$ with scaling function $b(t^{1/2})$. Observe that for $j=0$ we removed just $\bfT^m(\bfC_{\leq 0}) = \{\zero \}$ from $\RZ$ and obtained that $\bfX^m$ was regularly varying with a limit measure concentrating on $\bfT^m(\bfC_{=1})$ which is a very small part of the entire state space $\RZ \setminus \{ \zero \}$. Now, on also removing the support of $\nu^{m,(0)}$ ie. $\bfT^m(\bfC_{=1})$ from the state space we obtained a new regular variation property for $\bfX^m$ on a smaller state space $\RZ \setminus \bfT^m(\bfC_{\leq 1})$ with a finer scaling function $b(t^{1/2})$. This regular variation property was in some sense hidden by the cruder scaling that we used for the larger state space. This is a typical example of hidden regular variation. For a more expository account on such a nested sequence of regular variation properties in the case of i.i.d.~ heavy tailed random variables, see \cite{lindskog:resnick:2013} (Section 4.5).

In fact we have an increasing sequence of cones,
\begin{align*}
\bfT^m(\bfC_{\leq 0}) \subset \bfT^m(\bfC_{\leq 1}) \subset \ldots \bfT^m(\bfC_{\leq j}) \subset \ldots, 
\end{align*}
a sequence of non-zero measures $\nu^{m,(j)} , j \geq 0$ where $\nu^{m,(j)}$ is supported on\\  $\bfT^m(\bfC_{\leq j+1}) \setminus \bfT^m(\bfC_{\leq j}) $ and a sequence of decreasing scaling functions
\begin{align*}
b(t) > b(t^{1/2}) > \ldots b(t^{1/(j+1)}).\ldots
\end{align*}
such that $\bfX^m $ is regularly varying on $\RZ \setminus \bfT^m(\bfC_{\leq j})$ with limit measure $\nu^{m,(j)}$ and scaling function $b(t^{1/(j+1)})$.
Thus by removing more and more of the state space and using finer and finer scaling functions we are able to get a more detailed picture of the extremal properties of $\bfX^m$.
\end{remark}

\subsection{Regular variation of the MA($\infty$) process}
 As mentioned before the map $\bfT^{\infty}$ is only well-defined $P$-almost surely. For each $j > 0$, $\bfT^{\infty}(\bfC_{\leq j})$ is not closed, even though  $\bfT^{\infty}$ is well-defined on each $\bfC_{\leq j}$. This prevents us from proving a result implying hidden regular variation of $\bfX$ as in  Theorem \ref{tmconvergence} for $\bfX^m$. However, the fact that $\bfT^{\infty}(\{\mathbf{0}_{\infty}\}) = \{\mathbf{0}_{\infty}\}$ and the use of (C4) and interpreting $\bfX$ as the limit of $\bfX^m$ as $m \to \infty$, allows us to prove the following result. The proof, except for technical details, is similar in spirit to \cite{resnick:2007} (Theorem 3.5) or \cite{billingsley:1999} (Theorem 3.2).
\begin{theorem}\label{tinftyconvergence}
$\mu_t^{(0)} \circ (\bfT^{\infty} )^{-1} \to \mu^{(0)} \circ (\bfT^{\infty} )^{-1} =  \nu^{(0)}$ in $\bfM(\RZ \setminus \{\mathbf{0}_{\infty}\}) = \bfM(\bfO_0)$ or equivalently
\begin{align*}
 t P \left[ \bfX/b(t) \in \cdot \right] \to  \sum_{i=\infty}^{\infty} \int \mathbbm{1} \left\{ \bfT^{\infty}(z_i \mathbf{e}_i)\in \cdot \right\} \nu_\alpha (\ud z_i).
\end{align*}
\end{theorem}
\begin{remark}
(i) Theorem \ref{tinftyconvergence} implies that $\bfX$ is regularly varying on  $\RZ \setminus \zero$ with limit measure $\nu^{(0)}$ and scaling function $b(\cdot)$. The limit measure $\nu^{(0)}$ can also be expressed in the following way emphasizing the fact that its support is on $\bfTinf(\bfC_{=1})$ and it is indeed non-zero,
\begin{align}
& \nu^{(0)}(\ldots, \ud z_1, \ud z_0, \ud z_{-1}, \ldots) \nn\\
&\qquad=  \sum_{i=-\infty}^{\infty}  \left( \prod_{k < i \text{ or } \psi_{k-i} = 0}\bep_0 (\ud z_k) \right)    \left( \prod_{ k \geq i \text{  and } \psi_{k-i} > 0 }\nu_{\alpha}\left (\frac{\ud z_k}{\psi_{k-i}} \right) \right). \nn
\end{align}
Also we have that any $k \in \bbZ$, we have that 
\begin{align*}
t P[ X_k > b(t)x] \to  \left( \sum_{l=0}^{\infty} \psi_l^{\alpha} \right) x^{-\alpha} \text{ for } x>0.
\end{align*}
\\
(ii) It is also instructive to compare Theorem \ref{tinftyconvergence} to Theorem 2.4 in \cite{davis:resnick:1985} where a point process version of the same result was obtained.
\end{remark}

Before proving Theorem \ref{tinftyconvergence} we prove two technical lemmas. Set $\sum_{j=0}^{\infty} \psi_j = S$ which is finite by (C1).

\begin{lemma}\label{lem:pseudocontinuity}
For any $\gamma > 0$, 
\begin{align*}
\lim_{n \to \infty} \limsup_{t \to \infty} \mutz (\{ \bfz : \dZ(\bfT^{\infty}(\bfz),\zero)>\gamma,  \dZ(\bfz,\zero) < \delta_n\}) = 0
\end{align*}
where $\delta_n = 2^{-(n+1)}/S$.
\end{lemma}
\begin{proof}
Fix $\gamma > 0$ and let $M>0$ be such that $2\sum_{|i| \geq M} 2^{-(|i|+1)} < \gamma/2$. Then  we have
\begin{align}
&(\{ \bfz : \dZ(\bfT^{\infty}(\bfz),\zero)>\gamma,  \dZ(\bfz,\zero) < \delta_n\} \nn \\
&\quad \subset \left\{ \bfz : \sum_{|i|<M} \frac{ \bfTinf_i(\bfz) \wedge 1}{2^{|i|+1}} > \gamma/2,   \sum_{|i|<n} \frac{ \bfz_i \wedge 1}{2^{|i|+1}} < \delta_n \right\}\nn\\
&\quad \subset \left\{ \bfz : \bigvee_{|i|<M} \bfTinf_i(\bfz)> \gamma,   \bigvee_{|i|<n} z_i  < 2^n\delta_n \right\}\nn\\
&\quad \subset \bigcup_{|i|<M} \left\{ \bfz : \sum_{l=0}^{\infty} \psi_l z_{i-l} > \gamma,   \bigvee_{|i|<n} z_i  < 2^n\delta_n \right\}\nn\\
&\quad \subset \bigcup_{|i|<M}\left\{ \bfz :  (\sum_{l=0}^{\infty} \psi_l)(\bigvee_{|i|<n} z_i) + \sum_{l > i+n} \psi_l z_{i-l}> \gamma,   \bigvee_{|i|<n} z_i  < 2^n\delta_n \right\}\nn\\
&\quad \subset \bigcup_{|i|<M}\left\{ \bfz :   \sum_{l > i+n} \psi_l z_{i-l}> \gamma - S2^n\delta_n \right\}\nn\\
&\quad \subset \bigcup_{|i|<M}\left\{ \bfz :   \sum_{l > i+n} \psi_l z_{i-l}> \gamma/2 \right\}\nn
\end{align}
for large enough $n$. So we have that 
\begin{align}
&\lim_{n \to \infty} \limsup_{t \to \infty} \mutz (\{ \bfz : \dZ(\bfT^{\infty}(\bfz),\zero)>\gamma,  \dZ(\bfz,\zero) < \delta_n\}) \nn \\
&\quad \leq  \sum_{|i|<M}\lim_{n \to \infty} \limsup_{t \to \infty} \mutz  \left( \left\{ \bfz :   \sum_{l > i+n} \psi_l z_{i-l}> \gamma/2 \right\} \right) \nn \\
&\quad \leq  \sum_{|i|<M}\lim_{n \to \infty} \limsup_{t \to \infty} tP \left[\left\{ \bfz :   \sum_{l > i+n} \psi_l Z_{i-l}> b(t)\gamma/2 \right\} \right] \nn \\
&\quad = 0. \nn
\end{align}
The last line follows from (C4).
\end{proof}

\begin{lemma}\label{lem:d(tm,tinfty)>e}
For any $\beta > 0$, 
\begin{align*}
\lim_{m \to \infty} \limsup_{t \to \infty} \mutz (\{ \bfz : \dZ(\bfT^{\infty}(\bfz),\bfT^{m}(\bfz))>\beta\}) = 0
\end{align*}
\end{lemma}
\begin{proof}
Fix $\beta > 0$ and let $M>0$ be such that $2\sum_{|i| \geq M} 2^{-(|i|+1)} < \beta/2$. Then  we have
\begin{align}
&\{ \bfz : \dZ(\bfT^{\infty}(\bfz),\bfT^{m}(\bfz))>\beta\} \nn\\
&\subset \left\{ \bfz :  \sum_{|i|<M} \frac{ |\bfTinf(\bfz) - \bfTm(\bfz)|_i \wedge 1}{2^{|i|+1}} > \beta/2 \right\} \nn \\
&\subset \left\{ \bfz :  \sum_{|i|<M} \frac{ \sum\limits_{l>m+1} \psi_l z_{i-l} \wedge 1}{2^{|i|+1}} > \beta/2 \right\} \nn \\
&\subset \left\{ \bfz :  \bigvee_{|i|<M}  \sum\limits_{l>m+1} \psi_l z_{i-l}  > \beta \right\} \nn \\
& \subset \bigcup_{|i|<M} \left\{ \bfz :  \sum\limits_{l>m+1} \psi_l z_{i-l}  > \beta \right\}. \nn 
\end{align}
As in Lemma \ref{lem:pseudocontinuity}, we have
\begin{align*}
&\lim_{m \to \infty} \limsup_{t \to \infty} \mutz (\{ \bfz : \dZ(\bfT^{\infty}(\bfz),\bfT^{m}(\bfz))>\beta\}) \nn \\
&\quad \leq \sum_{|i|<M} tP \left[\left\{ \bfz :  \sum\limits_{l>m+1} \psi_l Z_{i-l}  > b(t) \beta \right\} \right], \nn 
\end{align*}
which is 0 by (C4).
\end{proof}

\begin{proof}[Proof of Theorem \ref{tinftyconvergence}]
By Theorem \ref{portthm}, it is enough to show that for any uniformly continuous $f \in \scrC(\bfO_0)$, we have that $\int f \; \ud \mutj \circ (\bfT^{\infty} )^{-1} \to \int f \; \ud \muj \circ (\bfT^{\infty} )^{-1}$. Fix any such $f$ and set $\mathbb{F} = \{ \mathbf{z} \in \RZ: f(\mathbf{z}) >0\}$. Since $f \in \scrC(\bfO_0)$, we may assume that $\dZ(\mathbb{F},\zero) > \gamma > 0$ and $\sup_{\bfz \in \RZ} f(\bfz) = 1$. Let $\omega_f(\cdot)$ be the modulus of continuity of $f$. 

Fix $\vep > 0$.  By lemma \ref{lem:pseudocontinuity} we can find $\mathbb{G}$ such that $\mathbb{G} = \{ \bfz \in \RZ: \dZ(\bfz,\zero) > \delta\} $ for some $\delta > 0$, $\muz(\partial \, \bfG) = 0$ and $\lim_{t \to \infty} \mutz(\mathbb{F} \setminus \bfTinf(\mathbb{G})) = \muz(\mathbb{F} \setminus \bfTinf(\mathbb{G})) < \vep$.
Then
\begin{align}
&\left|\int f \; \ud \mutz \circ (\bfT^{\infty} )^{-1} - \int f \; \ud \muz \circ (\bfT^{\infty} )^{-1}\right| \nn \\
&\quad = \left|\int_{\bfTinf(\bfz) \in \bfF} f\circ \bfT^{\infty}(\bfz) \;  \mutz (\ud\bfz) - \int_{\bfTinf(\bfz) \in \bfF} f \circ \bfT^{\infty}(\bfz)  \;   \muz (\ud \bfz) \right| \nn \\
&\quad \leq \left|\int_{\mathbb{G}} f\circ \bfT^{\infty}(\bfz) \;  \mutz (\ud\bfz) - \int_{\bf{G}} f \circ \bfT^{\infty}(\bfz)  \;   \muz (\ud \bfz) \right| \nn \\
&\qquad  \qquad + \mutz(\bfF \setminus \bfTinf(\mathbb{G})) + \muz(\bfF \setminus \bfTinf(\mathbb{G})) . \nn
\end{align}
Since the last two terms are less than $\vep$ for large enough $t$, it suffices to show that the first term in the last line above goes to 0. 
\begin{align}
&\left|\int_{\mathbb{G}} f\circ \bfT^{\infty}(\bfz) \;  \mutz (\ud\bfz) - \int_{\bf{G}} f \circ \bfT^{\infty}(\bfz)  \;   \muz (\ud \bfz) \right| \nn \\
&\quad \leq \left|\int_{\mathbb{G}} f\circ \bfT^{\infty}(\bfz) \;  \mutz (\ud\bfz) - \int_{\bf{G}} f \circ \bfTm(\bfz)  \;   \mutz (\ud \bfz) \right| \nn \\
&\qquad + \left|\int_{\mathbb{G}} f\circ \bfTm(\bfz) \;  \mutz (\ud\bfz) - \int_{\bf{G}} f \circ \bfTm(\bfz)  \;   \muz (\ud \bfz) \right| \nn \\
&\qquad + \left|\int_{\mathbb{G}} f\circ \bfTm(\bfz) \;  \muz (\ud\bfz) - \int_{\bf{G}} f \circ \bfT^{\infty}(\bfz)  \;   \muz (\ud \bfz) \right| \nn \\
&\quad = I + II + III. \nn
\end{align}
We deal with $I, II$ and $III$ separately. 

Observe that 
\begin{align}
I &\leq \int_{\mathbb{G}} \left| f\circ \bfT^{\infty}(\bfz) -  f \circ \bfTm(\bfz) \right|  \mathbbm{1} \{  \dZ(\bfT^{\infty}(\bfz),\bfT^{m}(\bfz)) \leq \beta\} \;   \mutz (\ud \bfz) \nn \\
&\qquad+ \int_{\mathbb{G}} \left| f\circ \bfT^{\infty}(\bfz) -  f \circ \bfTm(\bfz) \right|  \mathbbm{1} \{  \dZ(\bfT^{\infty}(\bfz),\bfT^{m}(\bfz)) > \beta\} \;   \mutz (\ud \bfz) \nn \\
&\leq \omega_f(\beta) \mutz(\bfG) + 2\mutz (\{ \bfz : \dZ(\bfT^{\infty}(\bfz),\bfT^{m}(\bfz))>\beta\}) .\nn 
\end{align}
The first term above goes to 0 as $\beta \to \infty$ as $\mutz(\bfG)$ is finite for all large $t$ while the second term goes to 0 by Lemma \ref{lem:d(tm,tinfty)>e} by first letting $t \to \infty$ and then letting $m \to \infty$.

For any fixed $m$, $f \circ \bfTm$ is continuous on $\bfO_0$ and hence on $\bfG$ and so by part (ii) of Theorem \ref{portthm} and using Theorem \ref{tmconvergence} for $j=0$, we have that $II$ goes to 0 as $t \to \infty$ for every $m$.

To deal with $III$, first observe that for any $\bfz \in \bfC_{=1}$, $\lim_{m \to \infty} \bfT^m(\bfz) = \bfTinf(\bfz)$. To see this let $\bfz = \lambda \mathbf{e}_i$. Then
\begin{align*}
\dZ( \bfTm(\bfz), \bfT^{\infty}(\bfz) ) \leq \lambda \sum_{l=m+1}^{\infty} \psi _l,
\end{align*}
which goes to 0 as $m \to \infty$ as $\sum_{j=0}^{\infty} \psi_j  < \infty$ by (C1). Since $\muz$ is finite on $\bfG$ and concentrates on $\bfC_{=1}$ and $f$ is continuous and bounded, we have, by dominated convergence, that $III$ goes to 0 as $m \to \infty$.
\end{proof}
\begin{remark}
The entire exercise in this paper could have been carried out in modestly more generality by assuming that the i.i.d.~ sequence $(Z_i, i\in \bbZ)$ were real-valued and instead of \eqref{iidseqdef1} we assumed that $|Z_0|$ was regularly varying with tail index $\alpha>0$ and 
\begin{align*}
\lim_{t \to  \infty}\frac{ P[Z_0 > t]}{P[|Z_0| >t]}= p \text{ and } \lim_{t \to  \infty}\frac{ P[Z_0 <t]}{P[|Z_0| >t]}= 1- p.
\end{align*}
\end{remark}

\bibliographystyle{amsplain}
\bibliography{bibfile}
\end{document}